\documentclass[12pt,a4paper]{amsart}
\usepackage{amssymb,color,currfile,pdfpages}
\usepackage[urlcolor=blue,colorlinks=true]{hyperref}

\usepackage[english]{babel}
\usepackage{verbatim,here}
\usepackage[T1]{fontenc}
\usepackage{epstopdf}
\usepackage{floatflt,graphicx}
\usepackage{color,xcolor}
\usepackage{enumerate}
\usepackage{animate}
\usepackage{a4wide}
\usepackage{amsmath, amssymb}
\usepackage{amsthm}
\usepackage{float}
\usepackage{subfigure}
\usepackage{tikz}
\usepackage{amssymb,color,currfile,pdfpages}
\usetikzlibrary{calc,fadings,decorations.pathreplacing}
\newcounter{minutes}
\divide\time by 60
\newcounter{hours}
\multiply\time by 60 \addtocounter{minutes}{-\time}

\dedicatory{}
\commby{}
\theoremstyle{plain}
\newtheorem{thm}[equation]{Theorem}
\newtheorem{cor}[equation]{Corollary}
\newtheorem{lem}[equation]{Lemma}

\theoremstyle{definition}

\theoremstyle{remark}

\newtheorem{nonsec}[equation]{}

\numberwithin{equation}{section}

\newcommand{\beq}{\begin{equation}}
\newcommand{\eeq}{\end{equation}}
\newcommand{\ben}{\begin{enumerate}}
\newcommand{\een}{\end{enumerate}}
\newcommand{\bequu}{\begin{eqnarray*}}
\newcommand{\eequu}{\end{eqnarray*}}
\newcommand{\bequ}{\begin{eqnarray}}
\newcommand{\eequ}{\end{eqnarray}}
\newcommand{\B}{\mathbb{B}^2}

\newcommand{\sh}{\,\textnormal{sh}}

\renewcommand{\th}{\,\textnormal{th}}

\font\fFt=eusm10 
\font\fFa=eusm7  
\font\fFp=eusm5  
\def\K{\mathchoice
{\hbox{\,\fFt K}}
{\hbox{\,\fFt K}}
{\hbox{\,\fFa K}}
{\hbox{\,\fFp K}}}

\begin{document}
\thispagestyle{empty}
\def\thefootnote{}

\title[Formulas for the visual angle metric]{Formulas for the visual angle metric}
\author[M. Fujimura]{Masayo Fujimura}
\address{M. Fujimura: Department of Mathematics, National Defense Academy of Japan, Japan}
\email{masayo@nda.ac.jp
ORCID ID:  \url{http://orcid.org/0000-0002-5837-8167
}}
\author[R. Kargar]{Rahim Kargar}
\address{R. Kargar: Department of Mathematics and Statistics, University of Turku, FI-20014 Turku, Finland}
\email{rakarg@utu.fi
 ORCID ID:  \url{http://orcid.org/0000-0003-1029-5386}}

\author[M. Vuorinen]{Matti Vuorinen}
\address{M. Vuorinen: Department of Mathematics and Statistics, University of Turku, FI-20014 Turku, Finland}
\email{vuorinen@utu.fi
 ORCID ID:  \url{http://orcid.org/0000-0002-1734-8228}}

\date{}

\begin{abstract}
We  prove several new formulas for the visual angle metric of the unit
disk in terms of the hyperbolic metric and apply these to prove a sharp Schwarz
lemma for the visual angle metric under quasiregular mappings.

\end{abstract}

\keywords{ hyperbolic metric, visual angle metric, conformal mapping, quasiconformal mapping}
\subjclass[2020]{30C62, 51M09, 51M15}

\maketitle

\footnotetext{\texttt{{\tiny File:~\jobname .tex, printed: \number\year-%
\number\month-\number\day, \thehours.\ifnum\theminutes<10{0}\fi\theminutes}}}
\makeatletter

\makeatother

\section{Introduction}

During the past few decades, various intrinsic metrics of planar domains
have become
important tools in geometric function theory, for instance in the study
of quasiconformal
mappings \cite{gh, hkv,h}.  These metrics, defined  in a general domain, on the one hand, share some of the properties
of the hyperbolic metric of the unit disk and on the other hand, they are
simpler  than
the hyperbolic metric. Intrinsic metrics are usually not conformally
invariant, but have some kind of quasi-invariance properties under
subclasses of conformal
maps, e.g. under translations or M\"obius transformations.

We study here one such intrinsic metric, the visual angle metric,
introduced in \cite{klvw} and further studied in \cite{hkv,hvw, wv}.
Let $G$ be a proper subdomain of $\mathbb{R}^n$ such that $\partial G$
is not a proper subset of a line. The \textit{visual angle metric}
for $a,b \in G$ is defined by
\begin{equation}\label{vamDef}
  v_G(a,b) = \sup \{ \alpha : \alpha=\measuredangle (a,z,b),
  z\in \partial{G} \}.
\end{equation}

Finding the concrete values of $v_G$ leads to minimization algorithms even
in the simplest case when $G$ is the unit disk $\mathbb{B}^2$  because no
formulas are known.  Our main
results are the following three theorems, which give explicit formulas for
$v_G$ when $G=\mathbb{B}^2.$
Let the line through $a,b \in \mathbb{C}$ be denoted by $L[a,b]$. Let $S(a,r)=\{b\in\mathbb{R}^n:|a-b|=r\}$ be the circle centered at $a\in\mathbb{R}^n$ with radius $r>0$. The unit circle is defined by $S(0,1)$.

The first theorem provides a geometric construction for the extremal point $z$.

\begin{thm} \label{main1}
Let  $a,b \in\mathbb{B}^2$ with $|a|\neq|b|$ and $0 \notin
L[a,b].$ Then
\[
 v_{\mathbb{B}^2}(a,b) = \max \{\measuredangle (a,z_1,b),
 \measuredangle (a,z_2,b) \},
\]
where $z_1$ and $z_2$ are the points of intersection of the unit circle
and an orthogonal circle
\[
S(0,1)\cap S(c,\sqrt{|c|^2-1}), \quad c=
\displaystyle{\frac{a(1-|b|^2)- b(1-|a|^2)}{|a|^2-|b|^2}}.
\]
Moreover, $\{z_1,z_2\}= (1 \pm i \sqrt{|c|^2 -1})/\overline{c}.$
In the case $|a|=|b|$
\[
v_{\mathbb{B}^2}(a,b)= 2\, {\arctan } \left(\frac{|a-b|}{2-|a+b|}\right).
\]
\end{thm}

It is easily seen that the visual angle metric $v_{\mathbb{B}^2}$ is not
invariant under M\"obius automorphisms of the unit disk. Nevertheless, we prove another formula for  the visual angle metric
involving the   M\"obius invariant hyperbolic
metric  $\rho_{\mathbb{B}^2}.$ This is possible, because for given 
$a,b \in {\mathbb{B}^2}$,
we have  $v_{\mathbb{B}^2}(a,b)= v_{\mathbb{B}^2}(h(a),h(b))$, whenever
$h $ is an inversion  with $h({\mathbb{B}^2})= {\mathbb{B}^2}$ and with  $h({\mathbb{B}^2} \cap L[a,b])= {\mathbb{B}^2}\cap L[a,b].$
%

\begin{thm} \label{main2}
  For $a,b\in \mathbb{B}^2$ we have
  \begin{equation}\label{formula for vam-2}
    \tan \frac{v_{\mathbb{B}^2}(a,b)}{2}=
    \frac{(1+|m|)u}{1+\sqrt{1+(1-|m|^2)u^2}},
    \quad u={\rm sh} \frac{\rho_{\mathbb{B}^2}(a,b)}{2},
  \end{equation}
  where $m=(\overline{a}b - a\overline{b})/(2(\overline{a}- \overline{b}))$
 is the midpoint of the chord of  the unit disk containing
 the two points $a$ and $b$ and hence $|m|=d(L[a,b],\{0\}).$
\end{thm}

Our third main result yields a sharp quasiregular version of the Schwarz lemma for the visual angle metric.
This result seems to be new in the case of analytic functions.

\begin{thm} \label{main3} Let $f: \mathbb{B}^2 \to \mathbb{B}^2=f( \mathbb{B}^2)$
be a non-constant $K$-quasiregular mapping, where $K\geq1$. For  $a,b\in \mathbb{B}^2,$
let $m_1 $ and $m_2$ be the midpoints of the chords of the unit disk containing
$f(a), f(b)$ and $a,b,$ respectively.
Then  we have
  \begin{equation*} 
    \tan \frac{v_{\mathbb{B}^2}(f(a),f(b))}{2} \le
    2^{1-1/K}\,c\,
    \left(\, \tan \frac{v_{\mathbb{B}^2}(a,b)}{2} \right)^{1/K}, \quad
    c=\sqrt{\frac{1+|m_1|}{1-|m_1|}}  \frac{1}{(1+|m_2|)^{1/K}},
  \end{equation*}
  with equality for $K=1$, and  $m_1= m_2=0.$
\end{thm}

In a later work, these theorems will be applied  to prove
inequalities between the
visual angle metric and the ubiquitous quasihyperbolic and distance ratio
metrics \cite{gh,hkv} and some other intrinsic metrics \cite{ha} such as
the Hilbert metric \cite{p, rv}.

\section{Preliminary results}\label{Sec2}

We will give here some formulas about the geometry of
lines and triangles, on which our later work is based.

\begin{nonsec}{\bf Geometry and complex numbers.}
The extended complex plane $\overline{\mathbb{C}}=
{\mathbb{C}} \cup \{\infty\}$ is identified with the
Riemann sphere via the stereographic projection.
Let $L[a,b]$ stand for the line through $a$ and $b\,(\neq a)$.
For distinct points $a,b,c,d \in {\mathbb{C}}$ such that the lines
$L[a,b]$ and  $L[c,d]$ have a unique
point $w$ of intersection, let
$$
w=LIS[a,b,c,d] = L[a,b]\cap L[c,d].
$$
This point is given by
\begin{equation}{\label{LIS}}
w=LIS[a,b,c,d] =\frac{u}{v},
\end{equation}
with (see e.g. \cite[Ex. 4.3(1), p. 57 and p. 373]{hkv})
\begin{equation}\label{Lineint}
\begin{cases}
{\displaystyle
{u= (\overline{a}b -a  \overline{b})(c-d)-(a-b) ( \overline{c}d -c  \overline{d}); }}&\\
{\displaystyle {v = ( \overline{a}- \overline{b})(c-d)-(a-b) (\overline{c} - \overline{d}).}}&
\end{cases}
\end{equation}

Let $C[a,b,c]$  be the  circle  through distinct noncollinear points $a$, $b$,
and $c$. The formula \eqref{LIS} gives easily the formula for the center
$m(a,b,c)$ of $C[a,b,c]$. For instance, we can find two points on
the bisecting normal to the side $[a,b]$ and another
two points on the bisecting normal to the side $[a,c]$ and then apply
\eqref{LIS}  to get $m(a,b,c).$
In this way we see that the center $m(a,b,c)$ of $C[a,b,c]$  is
\begin{equation}\label{mfun}
m(a,b,c)=\frac{ |a|^2(b-c) +  |b|^2(c-a) +  |c|^2(a-b) }
{a(\overline{c}-\overline{b}) +b(\overline{a}-\overline{c})+ c(\overline{b}-\overline{a})}.
\end{equation}


We sometimes use the notation $a^* = a/|a|^2 = 1/\overline{a}$
for $a\in {\mathbb C} \setminus \{0\}$. The reflection of a point $z$ in
the line through two distinct points $a$, and $b$ is given by
\begin{equation}\label{myrefl}
w(z) = \frac{a-b}{ \overline{a}- \overline{b}}\, \overline{z} -\frac{a \overline{b}- \overline{a}b}{ \overline{a} -  \overline{b}}.
\end{equation}

\end{nonsec}

\begin{nonsec}{\bf Proposition.}\label{myprop} {\it
Let $0<r <s<1$, $r_1= 1/r$ and $s_1=1/s$. Also, let $c \in \mathbb{B}^2$ with
${\rm Re}\, c =(s+r)/2,$
$c_1\in \mathbb{C}$ with ${\rm Re}\, c_1 =(s_1+r_1)/2$, and $\arg c = \arg c_1.$
If $a \in S(c,|c-r|)$ and $b\in S(c_1,|c_1-r_1|),$ then }
$2\measuredangle(s,a,r)= 2\measuredangle(r_1,b,s_1)= \measuredangle(r,c,s).$
\end{nonsec}
We omit the simple proof with a reference to Figure \ref{FigVamProp}.
\begin{figure}
    \centering
\begin{tikzpicture}[scale=3.5]
   \draw[thick] (0,0) -- (2,0); 
   \node [black] at (0,-0.01) {\textbullet}; 
   \node[scale=1] at (-0.04,-0.1) {$0$}; 
   \draw[thick] (0,0) circle (0.5 cm); 
   \node[scale=1] at (-0.45,-0.45) {$\mathbb{B}^2$}; 
   \node [black] at (0.15,-0.01) {\textbullet}; 
   \node[scale=1] at (0.15,-0.1) {$r$}; 
   \node [black] at (0.5,-0.01) {\textbullet}; 
   \node[scale=1] at (0.55,-0.1) {$1$}; 
   \node [black] at (0.3,-0.01) {\textbullet}; 
   \node[scale=1] at (0.3,-0.1) {$s$}; 
   \node [black] at (0.225,0.14) {\textbullet}; 
   \node[scale=1] at (0.225,0.22) {$c$};
   \draw[thick] (0.225,0.1454545) circle (0.163652 cm); 
  \draw[thick] (1.235,0.808080555) circle (0.902019 cm); 
   \node [black] at (0.07,0.19) {\textbullet}; 
   \node[scale=1] at (0.05,0.27) {$a$}; 
   \draw[thick,red] (0.15,0) -- (0.07,0.2); 
   \draw[thick,red] (0.3,0) -- (0.07,0.2); 
   \draw[thick,dashed] (0.15,0) -- (0.225,0.1454545); 
   \draw[thick,dashed] (0.3,0) -- (0.225,0.1454545); 
   \node [black] at (1.64,-0.01) {\textbullet}; 
   \node[scale=1] at (1.66,-0.1) {$r_1$}; 
   \node [black] at (0.8333,-0.01) {\textbullet}; 
   \node[scale=1] at (0.8333,-0.1) {$s_1$}; 
   \node [black] at (1.25,0.79) {\textbullet}; 
   \node[scale=1] at (1.25,0.888080555) {$c_1$}; 
   \draw[thick] (0,0) -- (1.25,0.808080555); 
   \node [black] at (1.92,1.38) {\textbullet}; 
   \node[scale=1] at (1.97,1.45) {$b$}; 
   \draw[thick,blue] (0.8333,0) -- (1.92,1.39); 
   \draw[thick,blue] (1.64,0) -- (1.92,1.39); 
   \draw[thick, dashed] (0.8333,0) -- (1.25,0.808080555); 
   \draw[thick, dashed] (1.64,0) -- (1.25,0.808080555); 
\end{tikzpicture}
    \caption{Here $s_1=1/s, r_1=1/r.$ The key points are that the triangles
    $\triangle(r,s,c)$ and  $\triangle(s_1,r_1,c_1)$ are similar, and
     hence the angles $\measuredangle(r,a,s)$ and $\measuredangle(s_1,b,r_1)$
    are equal.}
    \label{FigVamProp}
\end{figure}
\begin{nonsec}{\bf M\"obius transformations.}\label{mymob}
A M\"obius transformation is a mapping of the form
$$z \mapsto \frac{az+b}{cz+d}\,, \quad a,b,c,d,z \in {\mathbb C}\,, ad-bc\neq 0.$$
The special M\"obius transformation
\begin{equation}\label{myTa}
T_a(z) = \frac{z-a}{1- \overline{a}z}, \quad a \in \B \setminus \{0\},
\end{equation}
maps the unit disk $\mathbb{B}^2$ onto itself with
$T_a(a) =0$, and $T_a( \pm a/|a|)= \pm a/|a|.$
In complex analysis, quadruples of points have a very special role:
the absolute cross-ratio of four points $a$, $b$, $c$, and $d$ in the complex
plane $\mathbb{C}$,
\begin{equation}\label{myaxrat}
|a,b,c,d| = \frac{|a-c||b-d|}{|a-b||c-d|},
\end{equation}
is invariant under M\"obius transformations.
\end{nonsec}

\begin{nonsec}{\bf Hyperbolic geometry.}
We recall some basic formulas and notation for hyperbolic geometry from \cite{b}.
The hyperbolic metric $\rho_{\mathbb{B}^2}$ is defined by
\begin{equation}\label{myrho}
\sh \frac{\rho_{\mathbb{B}^2}(a,b)}{2}=
\frac{|a-b|}{\sqrt{(1-|a|^2)(1-|b|^2)}} ,\quad a,b\in \mathbb{B}^2.
\end{equation}

For $a,b \in \B \setminus  \{ 0 \}$ let
$${\rm ep}(a,b)= T_{-b}(T_b(a)/|T_b(a)|).$$
This formula defines the endpoints ${\rm ep}(a,b)$ and  ${\rm ep}(b,a)$ on the
unit circle of the hyperbolic line through $a$ and $b.$
The hyperbolic metric also satisfies
$$
\rho_{\B}(a,b)= \log |{\rm ep}(a,b),a,b, {\rm ep}(b,a)|.
$$

The circle which is orthogonal to the unit circle and contains two
distinct points
$a,b\in \mathbb{C}$ is denoted by $C[a,b]$.
If $a,b \in {\mathbb{B}^2}$ are distinct points, then
$C[a,b]\cap \partial {\mathbb{B}^2} = \{a_*, b_*\}$
where the points are labelled in such a way that $a_*$, $a$, $b$, and $b_*$
occur in this order on  $C[a,b].$ Note that above we have
used the notation ${\rm ep}(a,b)=a_*$, and  ${\rm ep}(b,a)= b_*.$
We denote by $J[a,b]$ the hyperbolic geodesic segment joining
two distinct points $ a,b \in {\mathbb{B}^2}.$
Then $J[a,b] $ is a subarc of
$ C[a,b]\cap{\mathbb{B}^2}\,$ and
the hyperbolic line is $J^*[a_*,b_*]= C[a,b]\cap{\mathbb{B}^2}$.
\end{nonsec}


\begin{lem}\label{myinv} Suppose that
$a,b\in {\mathbb B}^2 \setminus \{0\},$ are two points non-collinear with $0$
and $ |a|\neq |b|.$  Then the inversion $h: \B\to \B$ with
$ h(a) = b$ is given by
\begin{equation}\label{myTab}
h(z)= \frac{c\overline{z} -1}{\overline{z}-\overline{c}},
\quad c=LIS[a,b, a^*, b^*]= \frac{a-b +
 a b ( \overline{a}- \overline{b})}{|a|^2-|b|^2},
\end{equation}
and $h$ maps the chord $L[a,b]\cap {\mathbb B}^2$ onto itself. Moreover,
the orthogonal circle $S(c,\sqrt{|c|^2-1})$
intersects the  unit circle at the points
$\{z_1,z_2\}=(1\pm i\,\sqrt{|c|^2-1} )/\overline{c}.$
\end{lem}

\begin{proof}
The above simple formula for  $c$  follows from the
formulas \eqref{LIS} and \eqref{Lineint} for the intersection of two lines as
$c= LIS[a,b,a^*, b^*]$ and checking $h(a)=b$ is a simple verification.
The points
$z_1, z_2$ are found by solving the equations $|z|^2=1$ and
$|z-c|^2 = |c|^2-1.$
\end{proof}


We start our discussion of the visual angle metric in the simple case of
points on the same radius of the unit disk.


\begin{figure}
    \centering
\begin{tikzpicture}[scale=3.5]
   \draw[thick] (0, 0) circle (1 cm); 
   \node [black] at (0,0) {\textbullet}; 
   \node[scale=1.1] at (-0.9,-0.9) {$\mathbb{B}^2$}; 
   \node[scale=1] at (-0.04,-0.1) {$0$}; 
   \node[scale=1] at (-0.02, -0.67) {$a$}; 
   \node[scale=1] at (0.7, -0.1) {$b$}; 
   \node[scale=1] at (0, -1.35) {$a^*$}; 
   \node[scale=1] at (-0.02, -1.1) {$a_*$}; 
   \node[scale=1] at (1.1, - 0.2) {$b_*$}; 
   \node[scale=1] at (1.6, -0.6) {$b^*$}; 
   \node[scale=1] at (0.17, -0.35) {$m$}; 
   \node[scale=1] at (-1, -1.8) {$c$}; 
   \draw[thick] (-0.9, -1.7) -- (0.6, -0.2); 
   \draw[thick] (-0.9, -1.7) -- (1.5, -0.5); 
   \draw[thick, blue] (-0.188701,-0.0112989) circle (0.810961 cm); 
   \draw[thick, blue] (0.294583, -0.494583) circle (0.424333 cm); 
   \draw[thick, dashed] (0.128571, -0.328571) circle (0.488647 cm); 
   \draw[thick, red] (-0.9, -1.7) ++(30.6:1.648cm) arc (30.6:93.6:1.648cm);
   \draw[thick] (0.825,-1.025) ++(79.5:0.855cm) arc (79.5:178.3:0.855cm);
   \node [black] at (0, -0.81) {\textbullet}; 
   \node [black] at (0.6, -0.21) {\textbullet}; 
   \node [black] at (0, -1.26) {\textbullet}; 
   \node [black] at (-0.029753, -1.01) {\textbullet}; 
   \node [black] at (0.982822, - 0.194558) {\textbullet}; 
   \node [black] at (1.5, -0.51) {\textbullet}; 
   \node [black] at (0.180959, -0.47) {\textbullet}; 
   \node [black] at (-0.9, -1.71) {\textbullet}; 
\end{tikzpicture}
\caption{The inversion $h(z)= ({c \overline{z}-1})/({\overline{z}-\overline{c}})$ with $c=LIS[a,b,a^*, b^*]$ maps the unit disk onto itself and the chord containing $a$ and $b$ onto itself. The point $m$ is the hyperbolic midpoint of $a$ and $b$, and the hyperbolic circle through $a$ and $b$ centered at $m$ is drawn with a dashed line.
}
\label{fig-vam02}
\end{figure}

\begin{nonsec}{ \bf The case of  radial points $v_{\B}(r,s)$, where $0<r<s<1$.}
\label{vBrad}
Writing,
\begin{equation} \label{mycd}
c= \frac{r+s}{2}, \quad d=\frac{\sqrt{(1-r^2)(1-s^2)}}{2}, \quad {\rm and}\quad c^2+d^2 =
\left(\frac{1+rs}{2}\right)^2,
\end{equation}
we easily see that the circle $S(c+id,(1-rs)/2)$ passes through the points
$r$ and $s$ and is internally tangent to $S(0,1)$ at the point
 \[ p=\left(\frac{r+s}{1+rs},
\frac{\sqrt{(1-r^2)(1-s^2)}}{1+rs} \right).\] Therefore
\begin{equation} \label{myvBrad}
v_{\B}(r,s)= {\arcsin} \frac{s-r}{{1-r s}}.
\end{equation}
{Moreover, the segment $[-i,p]$ bisects the angle
$\measuredangle(r, p, s)$ and the circle orthogonal
to the unit circle at the points $p$ and $\overline{p}$ passes through
the hyperbolic midpoint
\[
\frac{r+s}{1+rs+ \sqrt{(1-r^2)(1-s^2)}},
\]
 of the segment $[r,s].$}
\end{nonsec}

\begin{lem}{\rm (\cite[Lemma 3.10]{klvw})} \label{radial}
For $a,b \in \mathbb{B}^2$ collinear with $0$ we have
$$\tan v_{ \mathbb{B}^2}(a,b)= {\rm sh} \frac{\rho_{ \mathbb{B}^2}(a,b)}{2}.
$$
\end{lem}
\begin{figure}
    \centering
\begin{tikzpicture}[scale=3.5]
   \draw[thick] (0, 0) circle (1 cm); 
   \node[scale=1.1] at (-0.9,-0.9) {$\mathbb{B}^2$}; 
   \node[scale=1] at (-0.04,-0.1) {$0$}; 
   \node[scale=1] at (0.330085, -0.1) {$r$}; 
   \node[scale=1] at (0.729997, -0.1) {$s$}; 
   \draw[thick] (0,0) -- (0.729997, 0);
   \node[scale=1] at (0.93,0.6) {$p$}; 
   \node[scale=1] at (0.93,-0.6) {$\overline{p}$}; 
   \draw[thick] (0,0) -- (0.854243,0.519875); 
   \draw[thick] (0.729997, 0) -- (0.854243,0.519875); 
   \draw[thick] (0.330085, 0) -- (0.854243,0.519875); 
   \draw[thick, red] (0, -1) -- (0.854243,0.519875); 
   \node[scale=1] at (0,-1.1) {$-i$}; 
   \draw[thick, blue, dotted] (0.530041,0.322572) circle (0.37952 cm); 
   \draw[thick, dashed] (1.17063, 0) ++(121:0.61cm) arc (121:239:0.61cm);
\node[black, red] at (0.562048,-0.01) {\textbullet}; 
   \node [black] at (0,-0.01) {\textbullet}; 
   \node[black] at (0.330085, -0.01) {\textbullet}; 
   \node[black] at (0.729997, -0.01) {\textbullet}; 
   \node[black] at (0.854243,0.509875) {\textbullet}; 
   \node[black] at (0.854243, -0.53) {\textbullet}; 
   \node[black] at (0.530041, 0.312572) {\textbullet}; 
   \node[black] at (0,-1.01) {\textbullet}; 
\end{tikzpicture}
\caption{The segment $[-i, p]$ bisects the angle $\measuredangle(r, p, s)$. The dashed circle orthogonal to the unit circle at the points $p$ and $\overline{p}$ passes through the hyperbolic midpoint of the segment $[r,s]$. Note that ${\rm Re}\{p\}=(r+s)/(1+rs)={\rm th}((\rho_{\B}(0,r)+\rho_{\B}(0,s))/2)$.}
\label{fig:03}
\end{figure}



\section{A formula for the visual angle metric}
%
%

\begin{nonsec}{\bf Angle bisection property.}
The next theorem, on the other hand, generalizes the above observations connected with the visual metric of two points on the same radius. On the other hand, it also yields a proof of Theorem \ref{main1}.
\end{nonsec}

\begin{figure}[!ht]
\centering
\subfigure[]{
\begin{tikzpicture}[scale=2]
   \draw[thick] (0, 0) circle (1 cm); 
   \node[scale=1] at (-0.9,-0.9) {$\mathbb{B}^2$}; 
   \draw[thick] (1.4, 0.45) circle (1.07819 cm); 
   \draw[thick, dashed] (0.253821, 0.543668) circle (0.4 cm); 
   \draw[thick] (0.032249, 1.08856) -- (1.4, 0.45); 
   \draw[thick] (-0.621286, 0.955321) -- (1.4, 0.45); 
   \draw[thick] (0,0) -- (0.423035, 0.906113); 
   \draw[thick] (0.642183, 0.639454) -- (0.423035, 0.906113); 
   \draw[thick] (-0.0437754, 0.810944) -- (0.423035, 0.906113); 
   \draw[thick] (0.353999, 0.7115) -- (0.423035, 0.906113); 
   \node [black] at (0,-0.01) {\textbullet}; 
   \node[scale=1] at (-0.08,-0.15) {$0$}; 
   \node [black] at (1.4, 0.43) {\textbullet}; 
   \node[scale=1] at (1.5, 0.38) {$c$}; 
   \node [black] at (0.423035, 0.886113) {\textbullet}; 
   \node[scale=1] at (0.4, 1.08) {$v$}; 
   \node [black] at (0.353999, 0.7) {\textbullet}; 
   \node[scale=1] at (0.42, 0.6) {$u$}; 
   \node [black] at (0.253821, 0.533668) {\textbullet}; 
   \node [black] at (-0.0437754, 0.79) {\textbullet}; 
   \node[scale=1] at (-0.1, 0.92) {$a$}; 
   \node [black] at (0.642183, 0.619454) {\textbullet}; 
   \node[scale=1] at (0.72, 0.51) {$b$}; 
\end{tikzpicture}
    \label{Fig-Angbis01}
}
\hspace*{3mm}
\subfigure[]{
\begin{tikzpicture}[scale=3]
   \draw[thick] (0,0) ++(18:1cm) arc (18:126:1cm);
   \draw[thick, dashed] (0.253821, 0.543668) circle (0.4 cm);
   \draw[thick] (-0.621286, 0.955321) -- (1.4, 0.45); 
   \draw[thick] (0.032249, 1.08856) -- (1.4, 0.45); 
  \draw[thick] (0,0) -- (0.423035, 0.906113); 
   \draw[thick] (0.642183, 0.639454) -- (0.423035, 0.906113); 
   \draw[thick] (-0.0437754, 0.810944) -- (0.423035, 0.906113); 
   \draw[thick] (0.353999, 0.7115) -- (0.423035, 0.906113); 
   \draw[thick] (0.642183, 0.639454) ++(130:0.08cm) arc (130:167:0.08cm);
   \node[scale=0.7] at (0.51,0.72) {$\phi$}; 
   \draw[thick] (0.423035, 0.906113) ++(152:0.08cm) arc (152:190:0.08cm);
   \node[scale=0.7] at (0.3,0.915) {$\phi$}; 
      \node [black] at (0.423035, 0.896113) {\textbullet}; 
   \node[scale=1] at (0.47, 0.99) {$v$}; 
   \node [black] at (0.353999, 0.7015) {\textbullet}; 
   \node[scale=1] at (0.39, 0.62) {$u$}; 
   \node [black] at (-0.0437754, 0.79944) {\textbullet}; 
   \node[scale=1] at (-0.07, 0.9) {$a$}; 
   \node [black] at (0.642183, 0.629454) {\textbullet}; 
   \node[scale=1] at (0.72, 0.55) {$b$}; 
   \node [black] at (1.4, 0.44) {\textbullet}; 
   \node[scale=1] at (1.5, 0.4) {$c$}; 
   \node [black] at (0,-0.01) {\textbullet}; 
    \node[scale=1] at (0.0, 1.18) {$w$}; 
    \node [black] at (0.05,1.074) {\textbullet}; 
   \node [black] at (0.253821, 0.53668) {\textbullet}; 
\end{tikzpicture}
\label{Fig-Angbis02}
}
\caption[Fuji proof]
{\subref{Fig-Angbis01}: Angle bisection visualized.
 \subref{Fig-Angbis02}: Angle bisection visualized, detail. Here   $\measuredangle (w,v,a) = \measuredangle (v,b,a)=\phi$.}
 \label{fig-vam04RL}
\end{figure}


\begin{thm} \label{anglebis} For $a,b \in  \mathbb{B}^2$ non-collinear with
$0$ and with $|a|\neq |b|,$ the circle $S(c,\sqrt{|c|^2-1})$
centered at $c= LIS[a,b, a^*, b^*]$ is orthogonal to $\partial \mathbb{B}^2.$
Let $u=S(c,\sqrt{|c|^2-1}) \cap L[a,b]$ and let
$v  \in S(c,\sqrt{|c|^2-1}) \cap \partial \mathbb{B}^2$ be the point
maximizing the angle $\measuredangle (a,z,b)$ with $z\in S(0,1)$. Then
\[
v_{ \mathbb{B}^2}(a,b)= \measuredangle(a,v,b) \quad { and}\quad
\measuredangle(a,v,u)= \measuredangle(u,v,b).
\]

\end{thm}
\begin{proof}
    Let $\ell= L[a,b]$, $u$, and $v$ be as above and $w= c+ 1.5(v-c),$ see
    Figure \ref{Fig-Angbis02}.
   It is clear
  that the triangle $ \triangle(u,c,v)$ is an isosceles triangle. Therefore,
  $\measuredangle (c,v,u)= \measuredangle (c,u,v)$ holds.
  It also follows from the
  Alternate Segment Theorem that
  $\measuredangle (w,v,a) = \measuredangle (v,b,a)=\phi$.
  Considering the sum of the inner angles of $ \triangle(u,b,v)$,
  we find that
  $\measuredangle (b,v,u) = \pi-\measuredangle (u,b,v)-\measuredangle (v,u,b)$.
  Considering also the line $L[w,c]$, we have
  $\measuredangle (a,v,u) = \pi-\measuredangle (u,b,v)-\measuredangle (v,u,b)$.
  Hence, we see that $\measuredangle (a,v,u)=\measuredangle (b,v,u)$.
\end{proof}

It should be noticed that in the above proof,  $v_{ \mathbb{B}^2}(a,b)\neq
2v_{ \mathbb{B}^2}(a,u).$



\begin{nonsec}{\bf A functional identity between $\rho_{\mathbb{B}^2}(a,b)$
and $v_{\mathbb{B}^2}(a,b)$.}
We will next prove a new formula for $v_{\mathbb{B}^2}(a,b)$ and
give first an auxiliary lemma.
\end{nonsec}

\begin{lem}\label{lem-m1 and m2 inversion}
  Let $m_1,m_2\in \mathbb{B}^2$ be non-collinear with $0$, and let
  $|m_1|\neq |m_2|$. Then, there exists an inversion
  $h:\mathbb{B}^2\rightarrow \mathbb{B}^2=h(\mathbb{B}^2)$ with
  the following properties:
  \begin{equation}\label{two pro i and ii}
    i)\quad h(m_1)=m_2;\qquad ii)\quad h(L[m_1,m_2]\cap \mathbb{B}^2)=L[m_1,m_2]\cap \mathbb{B}^2.
  \end{equation}
\end{lem}
\begin{proof} The proof follows from Lemma \ref{myinv}.

\end{proof}
\begin{figure}
    \centering
\begin{tikzpicture}[scale=3.5]
   \draw[thick] (0,0) circle (1 cm); 
   \node[scale=1] at (-0.9,-0.9) {$\mathbb{B}^2$}; 
   \node [black] at (0,-0.01) {\textbullet}; 
   \node[scale=1] at (-0.04,-0.1) {$0$}; 
   \node [black] at (1,-0.01) {\textbullet}; 
   \node[scale=1] at (1.1,-0.1) {$1$}; 
   \node [black] at (0.55,-0.01) {\textbullet}; 
   \node[scale=1] at (0.55,-0.1) {$a$}; 
   \node [black] at (1.81818,-0.01) {\textbullet}; 
   \node[scale=1] at (1.89818,-0.1) {$a^*$}; 
   \draw[thick,black] (0,0) -- (1.81818,0); 
   \node [black] at (0.65,0.6-0.01) {\textbullet}; 
   \node[scale=1] at (0.624,0.7) {$b$}; 
   \node [black] at (0.830671,0.766773-0.01) {\textbullet}; 
   \node[scale=1] at (0.93, 0.87) {$b^*$}; 
   \draw[thick,black] (0,0) -- (0.830671,0.766773); 
   \node [black] at (0.695313, 0.871875-0.01) {\textbullet}; 
   \node[scale=1] at (0.69, 0.99) {$c$}; 
   \draw[thick,black] (1.81818,0) -- (0.695313, 0.871875); 
   \draw[thick,black] (0.55,0) -- (0.695313, 0.871875); 
   \node [black] at (0.905141, 0.425112-0.01) {\textbullet}; 
   \draw[thick, red] (0.905141, 0.425112-0.01) ++(143:0.08cm) arc (143:229:0.07cm);
   \node[scale=1] at (1.03, 0.38) {$z_1$}; 
   \draw[thick, red] (0.55,0) -- (0.905141, 0.425112); 
   \draw[thick, red] (0.65,0.6) -- (0.905141, 0.425112); 
   \draw[thick, blue, dashed] (0.628605, 0.295233) circle (0.305518 cm); 
   \draw[thick, dashed] (0.695313, 0.871875) circle (0.493584 cm); 
   \node [black] at (0.628605, 0.295233-0.01) {\textbullet}; 
   \node [black] at (0.213, 0.963) {\textbullet}; 
   \node[scale=1] at (0.14, 1.06) {$z_2$}; 
   \draw[thick, gray] (0.213, 0.963) ++(287:0.08cm) arc (280:322:0.07cm);
   \draw[thick, gray, dashed] (0.55,0) -- (0.213, 0.963); 
   \draw[thick, gray, dashed] (0.65,0.6) -- (0.213, 0.963); 
\end{tikzpicture}
    \caption{If $c= LIS[a,b,a^*,b^*]$
    and $z_1=S(c,\sqrt{|c|^2-1}) \cap S(0,1)$ is a point in the sector
    with vertex $c$ and sides $L[c,a]$ and $L[c, a^*],$
    then $v_{\B}(a,b) = \measuredangle(a,z_1,b).$
}\label{Fig11}
\end{figure}
\begin{nonsec}{\bf Proof of Theorem \ref{main1}.}
Consider first the case $|a|\neq|b|$ with $0\notin L[a,b].$
The inversion $h$ in Lemma \ref{myinv} satisfies $h(a)= b$ and hence
the triangles $ \triangle(c,a,d)$ and  $ \triangle(c,b,d)$ are similar where $d$ is as in
Figure \ref{Fig11}. By the proof of Theorem \ref{anglebis} the
circle through $a,d,b$ is internally tangent to the unit circle and
 $v_{\B}(a,b)= \measuredangle(a,d,b).$ The formula for the points
$z_1, z_2$ is given in Lemma \ref{myinv}.
In the case  $|a|=|b|,$ the formula follows easily by symmetry.
 \hfill $\square$
\end{nonsec}

\begin{nonsec}{\bf Remark.}\label{z1z2formula}
The two points $z_1$ and $z_2$ in Theorem \ref{main1} are given as solutions
to the equation:
\[
(\bar{a}{ \cdot }(1-|b|^2)-\bar{b}{ \cdot }(1-|a|^2){ \cdot }z^2
   -2{ \cdot }(|a|^2-|b|^2){ \cdot }z+a{ \cdot }(1-|b|^2)-b{ \cdot }(1-|a|^2)=0.
   \]
\end{nonsec}

From Theorem \ref{main1},
substituting $ c=({a(1-|b|^2)- b(1-|a|^2)})/({|a|^2-|b|^2})$,
$ r^2=|c|^2-1 $ into the equation $ |z-c|=r $
of the circle with center $ c $ and radius $ r $, we have
\begin{align*}
  &(z-c)(\overline{z}-\overline{c})-r^2\\
    & = \dfrac{(|a|^2-|b|^2)z\overline{z}
        -(\overline{a}(1-|b|^2)-\overline{b}(1-|a|^2))z
        -(a(1-|b|^2)-b(1-|a|^2))\overline{z}+|a|^2-|b|^2}{|a|^2-|b|^2}\\
    &=0.
\end{align*}
Therefore, if $|a|\neq|b| $, the following holds:
\[
    (|a|^2-|b|^2)z\overline{z}
        -(\overline{a}(1-|b|^2)-\overline{b}(1-|a|^2))z
        -(a(1-|b|^2)-b(1-|a|^2))\overline{z}+|a|^2-|b|^2=0.
\]
Substituting $ \overline{z}=1/z $ into the above equality,
the intersection of the unit circle $ S(0,1) $ and the circle
$ S(c,\sqrt{|c|^2-1}) $ is obtained as the solution of the following equation
\begin{equation}\label{eq:v}
  (\overline{a}(1-|b|^2)-\overline{b}(1-|a|^2))z^2
     -2(|a|^2-|b|^2)z+a(1-|b|^2)-b(1-|a|^2)=0.
\end{equation}

\bigskip

On the other hand, we consider the line $ L$ and a circle $C$
both passing through points $a$ and $b$.
The circle $C$ is divided into two arcs by $L$.
On each arc, as $z$ ranges over each arc,
the angle $ \measuredangle(a,z,b) $ takes a constant value
from the inscribed angle theorem.
The larger the radius of the circle $C$,
the smaller the angle $ \measuredangle(a,z,b) $.
There exist two circles that have chord $ [a,b] $ and
are inscribed in the unit circle.
Let $\widetilde{C} $ be the smaller radius of them.
The intersection point $ q $ of $ \widetilde{C} $ and the unit circle
gives the visual angle metric for $ a,b $,
i.e. $ v_{\mathbb{B}^2}(a,b)= \measuredangle(a,q,b) $.
In fact, for $ c_0=m(a,b,p) $, the circle $ \widetilde{C} $ is written as
$ |z-c_0|=|q-c_0| $.
Since $ q $ is a point on the unit circle, $ q\overline{q}=1 $ holds.
Therefore, we have $ |z-c_0|^2-|q-c_0|^2=P/Q=0 $, where
\begin{align*}
   P&=((\overline{a}-\overline{b})q^2
       +(\overline{b}a-\overline{a}b)q-a+b)z\overline{z}
      +(-(\overline{a}\overline{b}(a-b)+\overline{a}
       -\overline{b})q+\overline{a}a-\overline{b}b)z \\
    & \quad  +((-\overline{a}a+\overline{b}b)q^2
      +((\overline{a}-\overline{b})ab+a-b)q)\overline{z}
      +\overline{b}\overline{a}(a-b)q^2
    +(-\overline{b}a+\overline{a}b)q+(-\overline{a}+\overline{b})ba,
\end{align*}
and
\begin{equation*}
  Q=(\overline{a}-\overline{b})q^2+(\overline{b}a-\overline{a}b)q-a+b.
\end{equation*}
The intersection points of $\widetilde{C} $ and the unit circle
are obtained as solutions to the equation substituting $ \overline{z}=1/z $
into $ P=0 $. Then, we have,
\[
   (z-q)\big(((\overline{a}\overline{b}(a-b)+\overline{a}-\overline{b})q
          -|a|^2+|b|^2)z+(-|a|^2+|b|^2)q+(\overline{a}-\overline{b})ab
        +a-b\big)=0.
\]
As $ q $ is the point of tangency, the above equation has a
double root $ z=q $.
Hence, we have
\[
  (\overline{a}(1-|b|^2)-\overline{b}(1-|a|^2))q^2
   -2(|a|^2-|b|^2)q+a(1-|b|^2)-b(1-|a|^2)=0.
\]
This equation is equivalent to \eqref{eq:v}.

\begin{nonsec}{\bf Proof of Theorem \ref{main2}.}
If $a,b\in \mathbb{B}^2$ are collinear with $0$,
the proof follows from Lemma  \ref{radial}. First, we consider the case
$|a|=|b|$, and denote $t=|a-b|/2$. Then $m=(a+b)/2,$  $s=|m|>0$ and
\begin{equation}\label{vB(a,b)}
    v_{\mathbb{B} ^2}(a,b)=2\,{\rm arctan }\frac{t}{1-s}\Leftrightarrow
     t=(1-s)\tan \frac{v_{\mathbb{B} ^2}(a,b)}{2}.
  \end{equation}

 Therefore, by \cite[p. 40]{b} and \eqref{vB(a,b)} we obtain
  \begin{equation*}
    {\rm sh} \frac{\rho_{\mathbb{B} ^2}(a,b)}{2}=\frac{2t}{1-s^2-t^2}=
\frac{2(1-s)\tan ({v_{\mathbb{B} ^2}(a,b)}/{2})}{1-s^2-(1-s)^2 \tan^2
({v_{\mathbb{B} ^2}(a,b)}/{2})}.
  \end{equation*}
Solving this equation for  $ \tan
(v_{\mathbb{B}^2}(a,b)/2)$ yields the desired result in this case.

Consider now the case $|a|\neq |b|$.
Observe that the formula for $m$ follows
from \eqref{myrefl}. Let {\tt hmid} be a point on $[a,b]$ with
\begin{equation*}
  \rho_{\mathbb{B}^2}(a,{\tt hmid})=\rho_{\mathbb{B}^2}({\tt hmid},b).
\end{equation*}

We shall apply Lemma
\ref{lem-m1 and m2 inversion} with $m_1={\tt hmid}$ and $m_2=m$ to
find an inversion $h:\mathbb{B}^2\rightarrow\mathbb{B}^2=h(\mathbb{B}^2)$.
 By Proposition
\ref{myprop} we also know that
$v_{\mathbb{B}^2}(a,b)=v_{\mathbb{B}^2}(h(a),h(b))$.
In conclusion,  because $|h(a)|=|h(b)|,$ this inversion reduces
the general case to the special case proved above.
The proof is now complete.\hfill $\square$
\end{nonsec}

\begin{figure}
    \centering
\begin{tikzpicture}[scale=3.5]
   \draw[thick] (0, 0) ++(0:1cm) arc (0:180:1cm);
   \draw[thick, dotted] (0.822336, 0.225248) circle (0.147373 cm); 
   \draw[thick, dotted] (0.0689818, 0.551855) circle (0.443851 cm); 
   \draw[thick] (1.19681, 0.0628984) ++(0:0.660543cm) arc (0:180:0.660543cm);
   \draw[thick, red] (-0.25058, 0.243822) -- (0.124035, 0.992278); 
   \draw[thick, red] (0.302887, 0.174639) -- (0.124035, 0.992278); 
   \draw[thick, red] (0.9,0.1) -- (0.964473, 0.264181); 
   \draw[thick, red] (0.716231, 0.122971) -- (0.964473, 0.264181); 
   \draw[thick] (-1, 0) -- (1.7, 0); 
   \draw[thick] (-0.943814, 0.330477) -- (1.7, 0); 
   \draw[thick, blue] (1.19681, 0.0628984) -- (0.0689818, 0.551855);
   \node [black, red] at (0.9,0.09) {\textbullet}; 
   \node[scale=1] at (0.9,0.045) {$a$}; 
   \node [black, blue] at (0.964473, 0.254181) {\textbullet}; 
   \node[scale=1] at (1.03, 0.31) {$v$}; 
   \node [black] at (0.716231, 0.112971) {\textbullet}; 
   \node[scale=1] at (0.67, 0.07) {$b$}; 
   \node [black, red] at (-0.25058, 0.233822) {\textbullet}; 
   \node[scale=1] at (-0.33, 0.15) {$h(a)$}; 
   \node [blue] at (0.124035, 0.982278) {\textbullet}; 
   \node[scale=1] at (0.12, 1.1) {$h(v)$}; 
   \node [black] at (0.302887, 0.164639) {\textbullet}; 
   \node[scale=1] at (0.35, 0.075) {$h(b)$}; 
   \node [black] at (0,-0.01) {\textbullet}; 
   \node[scale=1] at (-0.04, -0.1) {$0$}; 
   \node [black] at (1.7,-0.01) {\textbullet}; 
   \node[scale=1] at (1.75, -0.1) {$c$}; 
   \node [black] at (0.822336, 0.215248) {\textbullet}; 
   \node[scale=1] at (0.82, 0.31) {$d$}; 
   \node [black] at (0.0689818, 0.541855) {\textbullet}; 
   \node[scale=0.93] at (0.08, 0.65) {$h(d)$}; 
   \node [black] at (1.19681, 0.0528984) {\textbullet}; %
   \node [black] at (0.0261538, 0.199231) {\textbullet}; 
   \node[scale=1] at (0.02, 0.3) {$m$}; 
   \node [black] at (1, -0.01) {\textbullet}; 
   \node[scale=1] at (1, -0.1) {$1$}; 
   \node [black] at (0.829836, 0.09877) {\textbullet}; 
\end{tikzpicture}    \caption{The idea of the proof of Theorem \ref{main2} visualized.
    The key points are that the triangles
    $ \triangle(b,a,d)$ and  $ \triangle(h(a),h(b),h(d))$ are similar,
    and hence the angles $\measuredangle(h(b),h(v),h(a))$ and
    $\measuredangle(a,v,b)$
    are equal.
    }\label{Fig07}
\end{figure}


\medskip

Observe that in the case when $a$, $b$, and $0$ are collinear, Theorem \ref{main2}
reduces to Lemma  \ref{radial} and that Corollary \ref{vamSin} (2) gives
an equivalent form of Lemma  \ref{radial}.

\begin{cor} \label{vamSin} (1)  For $a,b\in \mathbb{B}^2$ and $m$ as
 in Theorem \ref{main2}, we have
  \begin{equation}\label{formula for vam}
    \sin v_{\mathbb{B}^2}(a,b)=
    \frac{(1+|m|)(1+\sqrt{1+(1-|m|^2)u^2})u}{1+\sqrt{1+(1-|m|^2)u^2}+(1+|m|)u^2},
    \quad u={\rm sh} \frac{\rho_{\mathbb{B}^2}(a,b)}{2}.
  \end{equation}
(2) For $m=0$ we have
\[
 \sin v_{\mathbb{B}^2}(a,b)=
    {\rm th} \frac{\rho_{\mathbb{B}^2}(a,b)}{2}= \frac{|a-b|}{\sqrt{|a-b|^2+(1-|a|^2)(1-|b|^2)}}.
\]
 \end{cor}

\begin{proof}
Both (1) and (2) follow from  Theorem \ref{main2} by simple manipulations.
\end{proof}

%

 \medskip

\begin{nonsec}\label{myProp2} {\bf Proposition. }{\it For $m \in (0,1)$,
$r >0$, and $u= {\rm sh}(r/2)$}
\[
(1+m)\, {\rm th}\frac{r}{4}\le
\frac{(1+m)\, u}{1+\sqrt{1+(1-m^2)\,u^2}} \le
\min \left\{ \frac{(1+m) u}{2},\sqrt{\frac{1+m}{1-m}}\,{\rm th} \frac{r}{4} \right\}.
\]
\end{nonsec}

\begin{proof} Observing $ {\rm th}(r/4)= u/(1+ \sqrt{1+u^2})$ and writing  $v$ for the middle term we have
\[
(1+m)\, {\rm th}\frac{r}{4}=\frac{(1+m)u}{1+\sqrt{1+u^2}} \le v \le
 \frac{(1+m)u}{\sqrt{1-m^2}(1+\sqrt{1+u^2})}=\sqrt{\frac{1+m}{1-m}}
 \,{\rm th}\frac{r}{4},
\]
and trivially $v \le (1+m)u/2.$
\end{proof}

The next corollary yields, as a special case, Theorem 3.11 of \cite{klvw}.

\begin{cor} \label{KLVW311}   For $a,b\in \mathbb{B}^2$ and, $m$ as
 in Theorem \ref{main2}, we have
  \begin{equation*} 
  (1+|m|) \th \frac{\rho_{\mathbb{B}^2}(a,b)}{4}
  \le \tan \frac{v_{\mathbb{B}^2}(a,b)}{2}
    \le \min \left\{ \frac{1+|m|}{2}\,{\rm sh} \frac{\rho_{\mathbb{B}^2}(a,b)}{2},
    \sqrt{\frac{1+|m|}{1-|m|}}
 \, \th \frac{\rho_{\mathbb{B}^2}(a,b)}{4} \right\}.
  \end{equation*}

 \end{cor}

\begin{proof} The proof follows from Theorem \ref{main2} and Proposition
\ref{myProp2}.
\end{proof}

For the proof of Theorem \ref{main3} we need some basic facts about quasiregular
mappings, see \cite{avv,hkv}. In particular, we use the quasiregular Schwarz lemma
in the following form with detailed information about the distortion function $\varphi_K$. 
For $r\in(0,1)$ and $K\in[1,\infty)$ the function $\varphi_K [0,1]\rightarrow [0,1]$ is defined by
\begin{equation*}
    \varphi_K(r)=\mu^{-1}(\mu(r)/K),
    \quad (\varphi_K(0)=0, \varphi_K(1)=1),
\end{equation*}
where $\mu:(0,1)\rightarrow(0,\infty)$ is the decreasing homeomorphism defined by
\begin{equation*}
    \mu(r)=\frac{\pi}{2}\frac{\K(\sqrt{1-r^2})}{\K(r)};\qquad \K(r)=\frac{\pi}{2}F(1/2,1/2;1;r^2),
\end{equation*}
and $F$ is the Gaussian hypergeometric function.

\begin{lem} \label{SLem} (1) Let $f:\mathbb{B}^2 \to \mathbb{B}^2$ be a $K$-quasiregular
mapping, where $K\geq1$, and $a,b \in \mathbb{B}^2.$ Then
\[
{\rm th} \frac{\rho_{\mathbb{B}^2}(f(a),f(b))}{2} \le
\varphi_K\left({\rm th} \frac{\rho_{\mathbb{B}^2}(a,b)}{2} \right)
\le 4^{1-1/K} \left({\rm th} \frac{\rho_{\mathbb{B}^2}(a,b)}{2} \right)^{1/K}.
\]
(2) The function $\varphi_K, K\ge 1,$ satisfies for $0<r<1$
\[
\frac{\varphi_K(r)}{1+ \sqrt{1-\varphi_K(r)^2 }}
 =\sqrt{\varphi_K\left(\left(\frac{r}{1+\sqrt{1-r^2}}\right)^2\right)}.
\]

\end{lem}

\begin{proof} (1) See \cite[Thm 16.2(1)]{hkv}.

\noindent
(2) See \cite[Theorem 10.5]{avv}.
\end{proof}

\begin{nonsec}{\bf Proof of Theorem \ref{main3}.}
Write $\rho'= \rho_{\mathbb{B}^2}(f(a),f(b))$, and $\rho= \rho_{\mathbb{B}^2}(a,b)$.
By Theorem \ref{main2},
Proposition \ref{myProp2} with $M=  \sqrt{({1+|m_1|})/({1-|m_1|})}$ and Lemma \ref{SLem}
\begin{equation*}
       \tan \frac{v_{\mathbb{B}^2}(f(a),f(b))}{2}\le
   M\,{\rm th}\frac{ \rho'}{4}
 =\frac{M
 \,{\rm th}\frac{\rho'}{2}}
 {1+ \sqrt{1-{\rm th}^2\frac{\rho'}{2}}}
  \le
\frac{M \,\varphi_K({\rm th}\frac{\rho}{2})}
 {1+ \sqrt{1-\varphi_K({\rm th}\frac{\rho}{2})^2}}=  M \,\sqrt{\varphi_K\left({\rm th}^2\frac{\rho}{4}\right)}\,
\end{equation*}
\begin{equation*}
\le  M
 \,\sqrt{\varphi_K\left(\left( \frac{1}{1+|m_2|}{\rm tan}\frac{v_{\mathbb{B}^2}(a,b)}{2}\right)^2\right)}
\le \frac{ M  2^{1-1/K}}{(1+|m_2|)^{1/K}} \left(
{\rm tan}\frac{v_{\mathbb{B}^2}(a,b)}{2} \right)^{1/K},
\end{equation*}
where in the second equality we have used this fact that if $r={\th} (\rho/2)$, then $r/(1+\sqrt{1-r^2})={\th} (\rho/4)$. $\hfill \square$
\end{nonsec}


\begin{cor} \label{vamMob} Let $T_w: {\mathbb{B}^2} \to {\mathbb{B}^2}$
be a M\"obius transformation as in \eqref{myTa}, let
  $a,b,w \in \mathbb{B}^2$, $|m_2|=d(L[a,b],0)$, and
  $[m_1|= d(L[T_w(a),T_w(b)],0).$
   Then
   \[
    \tan \frac{v_{\mathbb{B}^2}(T_w(a),T_w(b))}{2}\le c(m_1,m_2)
     \,\tan \,\frac{v_{\mathbb{B}^2}(a,b)}{2},\quad
     c(m_1,m_2)=\frac{1}{{1+|m_2|}} \sqrt{\frac{1+|m_1|}{1-|m_1|}},
   \]
with equality for $m_1=m_2=0.$
\end{cor}

\begin{proof} The proof follows from Theorem \ref{main3}.
\end{proof}

We next consider an evenly separated sequence of collinear points.

\begin{figure}
    \centering
\begin{tikzpicture}[scale=3.5]
   \draw[thick] (0, 0) circle (1 cm); 
   \node [black] at (0,-0.01) {\textbullet}; 
   \node[scale=1] at (-0.04, -0.1) {$0$}; 
   \node[scale=1.1] at (-0.9,-0.9) {$\mathbb{B}^2$}; 
   \node [black] at (1.7,-0.01) {\textbullet}; 
   \node[scale=1] at (1.74, -0.1) {$c$}; 
   \draw[thick] (-1,0) -- (1.7,0); %
   \draw[thick] (1.7,0) -- (-0.7,0.3);
   \node [black] at (0.9, 0.09) {\textbullet}; 
   \node [black] at (0.715752, 0.113031) {\textbullet}; 
   \node [black] at (0.309796, 0.163775) {\textbullet}; 
   \node [black] at (-0.244533, 0.233067) {\textbullet}; 
   \node [black] at (-0.656395, 0.284549) {\textbullet}; 
   \node [black] at (-0.845, 0.308125) {\textbullet}; 
   \node [black] at (-0.912209, 0.316526) {\textbullet}; 
   \node [black] at (-0.93395, 0.319244) {\textbullet}; 
   \node [blue] at (0.964427, 0.25435) {\textbullet}; 
   \node [blue] at (0.806286, 0.581526) {\textbullet}; 
   \node [blue] at (0.135677, 0.980753) {\textbullet}; 
   \node [blue] at (-0.624752, 0.770823) {\textbullet}; 
   \node [blue] at (-0.867122, 0.488095) {\textbullet}; 
   \node [blue] at (-0.923106, 0.374546) {\textbullet}; 
   \draw[thick, red] (0.9, 0.09) -- (0.964427, 0.26435); 
   \draw[thick, red] (0.715752, 0.123031) -- (0.964427, 0.26435); 
   \draw[thick, red] (0.715752, 0.123031) -- (0.806286, 0.591526); 
   \draw[thick, red] (0.309796, 0.173775) -- (0.806286, 0.591526); 
   \draw[thick, red] (0.309796, 0.173775) -- (0.135677, 0.990753); 
   \draw[thick, red] (-0.244533, 0.243067) -- (0.135677, 0.990753); 
   \draw[thick, red] (-0.244533, 0.243067) -- (-0.624752, 0.780823); 
   \draw[thick, red] (-0.656395, 0.294549) -- (-0.624752, 0.780823); 
   \draw[thick, red] (-0.656395, 0.294549) -- (-0.867122, 0.498095); 
   \draw[thick, red] (-0.845, 0.318125) -- (-0.867122, 0.498095); 
   \draw[thick, red] (-0.845, 0.318125) -- (-0.923106, 0.384546); 
   \draw[thick, red] (-0.912209, 0.326526) -- (-0.923106, 0.384546); 
   \draw[thick] (0.822104, 0.225339) circle (0.147572 cm); 
   \draw[thick] (0.544123, 0.399192) circle (0.325149 cm); 
   \draw[thick] (0.0754379, 0.550871) circle (0.443988 cm); 
   \draw[thick] (-0.418659, 0.523246) circle (0.329879 cm); 
   \draw[thick] (-0.736133, 0.422852) circle (0.151062 cm); 
   \draw[thick] (-0.873414, 0.363846) circle (0.0538309 cm); 
\end{tikzpicture}
\caption{Here  $\rho_{\mathbb{B}^2}(a_j,a_{j+1})={\rm const}$
    and also $v_{\mathbb{B}^2}(a_j,a_{j+1})$ is a constant.
    It follows from Proposition \ref{myprop} that the sectors with vertices
on the unit circle have equal angles.
   }\label{Fig07a}
\end{figure}
\begin{nonsec}{\bf Evenly separated collinear points.}
Consider a collinear sequence of points $a_j, j=1,2,3,\dots$ in the unit
disk with a constant hyperbolic distance $\rho_{\mathbb{B}^2}(a_j,a_{j+1})=const.$
It turns out that also  $v_{\mathbb{B}^2}(a_j,a_{j+1})$ is a constant--
the special case when the points are on the diameter $(-1,1)$ was discussed
in greater detail in Section \ref{Sec2}.
\end{nonsec}

%
%

\section{Additional identities for $v_{\mathbb{B}^2}(a,b)$ }

We give here another two analytic formulas, for $v_{\mathbb{B}^2}(a,b).$
The first formula is based on the recent explicit formula for the hyperbolic midpoint of two points, whereas for the second formula we have used symbolic computation. Both formulas are based on geometric ideas and are best used for computer work because some lengthy expressions will be needed.

We use the Ahlfors bracket notation $A[a,b]$ \cite[p. 38]{hkv}, for $a,b \in {\mathbb{B}^2}$
\begin{equation}\label{myahl}
A[a,b] =\sqrt{|a-b|^2+(1-|a|^2)(1-|b|^2) } = |1-a \overline{b}|.
\end{equation}

The formula for the hyperbolic midpoint is given in the following theorem.

\begin{thm}\label{myhmidp}  \cite{wvz}
For given $a,b \in \B$, the hyperbolic midpoint $z\in \B$ with
$\rho_{\B}(a,z)=\rho_{\B}(b,z)=\rho_{\B}(a,b)/2$ is given by
\begin{equation}\label{myzformua}
z= \frac{b(1-|a|^2) + a(1-|b|^2)}{1-|a|^2|b|^2 + A[a,b] \sqrt{(1-|a|^2)(1-|b|^2)}},
\end{equation}
where $A[a,b]$ is the Ahlfors bracket defined as \eqref{myahl}.
\end{thm}
\begin{figure}[!ht]
\centering
\subfigure[]{
\begin{tikzpicture}[scale=3.4]
   \draw[thick] (0, 0) circle (1 cm); 
   \node [black] at (0,0) {\textbullet}; 
    \node[scale=1] at (-0.9,-0.9) {$\mathbb{B}^2$}; 
   \node[scale=1] at (-0.07,-0.07) {$0$}; 
   \node[scale=1] at (0.15,0.3) {$a$}; 
   \node[scale=1] at (0.7,0.32) {$b$}; 
   \node[scale=1] at (0.385, 0.235) {$m$}; 
   \node[scale=1] at (0.9, 0.74) {$q_1$}; 
   \draw[thick] (0.511537, 0.453843) circle (0.316156 cm); 
   \draw[thick] (0.227361, -0.256596) circle (0.657166 cm); 
   \node[scale=1] at (0.76, -0.848) {$q_2$}; 
   \draw[thick, dashed] (0.65, -0.748457) arc (221:130:1); 
   \node [black] at (-0.125,0.98) {\textbullet};
    \node [black] at (0.978,0.21) {\textbullet};
    \draw[thick, blue] (0.978,0.22) arc (282:189:0.92);
   \node [black] at (0.205, 0.39) {\textbullet}; 
   \node [black] at (0.7, 0.19) {\textbullet}; 
   \node [black] at (0.467, 0.245) {\textbullet}; 
   \node [black] at (0.748031, 0.65) {\textbullet}; 
   \node [black] at (0.66, -0.77) {\textbullet}; 
   \node [black] at (0.511537, 0.453843) {\textbullet};
   \node [black] at (0.227361, -0.206596) {\textbullet};
\end{tikzpicture}
    \label{FigmyL}
}
\hspace*{3mm}
\subfigure[]{
\begin{tikzpicture}[scale=3.4]
   \draw[thick] (0, 0) circle (1 cm); 
    \node[scale=1] at (-0.9,-0.9) {$\mathbb{B}^2$}; 
   \draw[thick] (0.143379, 0.404352) circle (0.57098 cm); 
   \draw[thick] (-0.143379,-0.404352) circle (0.57098 cm); 
   \draw[thick, blue] (0.942502, -0.334201) -- (-0.942502, 0.334201); 
   \draw[thick, dashed] (-0.334201,-0.942502) -- (0.334201, +0.942502); 
   \node[scale=1] at (-0.38,-1.05) {$-Q$}; 
   \node[scale=1] at (0.38, 1.05) {$Q$}; 
   \node[scale=1] at (-0.66,0.11) {$T_m(a)$}; 
   \node[scale=1] at (0.65, -0.12) {$T_m(b)$}; 
   \node [black] at (0,-0.01) {\textbullet}; 
   \node [black] at (0.147, 0.404352) {\textbullet}; 
   \node [black] at (-0.14,-0.404352) {\textbullet}; 
   \node [black] at (-0.334201,-0.95) {\textbullet}; 
   \node [black] at (0.334201, 0.93) {\textbullet}; 
   \node [black] at (-0.355112, 0.12) {\textbullet}; 
   \node [black] at (0.355112, -0.133) {\textbullet}; 
\end{tikzpicture}
\label{FigmyR}
}
\caption[Fuji proof-tikz]
{\subref{Fig-Angbis01}: Before M\"obius transformation $T_m$.
 \subref{Fig-Angbis02}: After M\"obius transformation $T_m$.}
 \label{Fig2Mob}
\end{figure}

Theorem \ref{vamHmid} is essentially the same as \cite[ Theorem 3.2]{hvw}. Note, however, that by Theorem \ref{main1} we have now explicit formulas for the points $q_1$ and $q_2.$

\begin{thm}\label{vamHmid}  
For given, $a,b \in \B$, let $m\in \B$ be their hyperbolic midpoint.
Then
\begin{equation}\label{vamHformua}
v_{\B}(a,b) = \max\{ \measuredangle(a,q_1,b),\measuredangle(a,q_2,b) \},
\end{equation}
where $$q_1= T_m^{-1}(Q), \quad q_2= T_m^{-1}(-Q),$$
and $Q=i(T_m(a)-T_m(b))/|T_m(a)-T_m(b)|.$
\end{thm}
\begin{proof} Because  $|T_m(a)|=| T_m(b)|$ it is evident that the circles
through $T_m(a), T_m(b), Q$ and
 $T_m(a), T_m(b), -Q$ are the maximal circles through $T_m(a)$ and $ T_m(b)$
 contained in the unit disk. Therefore, one of the circles through
 $a$, $b$, and $q_1$ or $a$, $b$, and $q_2$ must be the maximal circle in $\B$ through $a$ and $b$.
\end{proof}

%
%



\begin{thm}\label{pThm}
For given $ a,b\in\mathbb{B}^2 $,
let $ q\in S(0,1) $ be the point that attains
$ v_{\mathbb{B}^2}(a,b)=\measuredangle(a,q,b) $ and set $ p=a+t(b-a)i $.
If a right triangle $ \triangle (a,b,p) $ is inscribed in
the circle passing through three points $a$, $b$, and $ p $,
then $t$ is given as the solution with the smaller absolute value of
the following equation
\begin{align*}
 & \big((a\overline{b} - \overline{a}b)^2 +
  4(\overline{a} - \overline{b})(a - b)\big)t^2
  +2(a\overline{b} - \overline{a}b)(a\overline{b} + \overline{a}b - 2)it
  \\
 & -\big((a\overline{b} + \overline{a}b)^2 -
      4(a\overline{a} + b\overline{b} - 1)\big) = 0.
\end{align*}
\end{thm}

\begin{proof}
Let $r$ be the radius of the inscribed maximal circle through $a$ and $b$.
Then,
$ | c | +r = 1$ and $c = (p + b)/2$ hold.
Eliminating $c$ from the above equations, we have $ f_1(r)=0 $, where
\[
  f_1(r)=(p +b)(\overline{p} +\overline{b})- 4(1 - r)^2.
\]
Since the triangle $ \triangle(a,b,p)$ is a right triangle inscribed in the circle
with center $ c $, we have $ f_2(r)=0 $, where

\begin{figure}
    \centering
\begin{tikzpicture}[scale=3.5]
   \draw[thick] (0, 0) circle (1 cm); 
   \node [black] at (0,-0.01) {\textbullet}; 
   \node[scale=1.1] at (-0.9,-0.9) {$\mathbb{B}^2$}; 
   \draw[thick] (0.511537, 0.453843) circle (0.316156 cm); 
   \draw[thick, dashed] (0, 0) -- (0.748031, 0.663663); 
   \draw[thick] (-0.245277, -0.713193) -- (0.323074, 0.707685); 
   \draw[thick] (0.227361, -0.256596) circle (0.657166 cm); %
   \draw[blue, thick] (0.775, 1.1125) ++(188:0.91cm) arc (188:283:0.91cm);
   \draw[red, thick] (1.41212, -0.0848485) ++(131:1cm) arc (131:221.9:1cm);
   \node[scale=1] at (-0.07,-0.07) {$0$}; 
   \node [black] at (0.2, 0.39) {\textbullet}; 
   \node[scale=1] at (0.09,0.43) {$a$}; 
   \node [black] at (0.227361, -0.256596) {\textbullet}; 
   \node [black] at (0.7, 0.19) {\textbullet}; 
   \node[scale=1] at (0.7,0.29) {$b$}; 
   \node [black] at (0.323074, 0.7) {\textbullet}; 
   \node[scale=1] at (0.323074, 0.8) {$p$}; 
   \node [black] at (-0.245277, -0.723193) {\textbullet}; %
   \node [black] at (0.511537, 0.44) {\textbullet}; 
   \node[scale=1] at (0.511537, 0.55) {$c$}; 
   \node [black] at (0.748031, 0.65) {\textbullet}; 
   \node[scale=1] at (0.848031, 0.763663) {$q$}; 
   \node [black] at (0.469014, 0.245) {\textbullet}; 
\end{tikzpicture}
\caption{$|c|+|p-c|=1$ yields an equation to find $p$ using RISA. Then $v_{\B}(a,b) = \measuredangle(a,q,b)= \measuredangle(a,p,b).$ }
\label{Fig08}
\end{figure}

\[
  f_2(r)=(a - b)(\overline{a}- \overline{b})
          + (a - p)(\overline{a} - \overline{p})- 4r^2.
\]
Eliminating $r$ from equations $ f_1=0 $ and $ f_2=0 $
by computing the following resultant
\[
    \mbox{resul}_r(f_1,f_2)=0,
\]
we have
\begin{align}\label{eq:resul}
 & (\overline{a}+\overline{b})^2p^2
  +2\big((a+b)(\overline{a}+\overline{b})-8\big)p\overline{p}
  -2\big((\overline{a}-\overline{b})(a\overline{b}+\overline{a}b-4)
  +2\overline{a}^2(a-b)\big)p \\ \notag
 &  +(a+b)^2\overline{p}^2
  -2\big((a-b)(a\overline{b}+\overline{a}b-4)
  +2a^2(\overline{a}-\overline{b})\big)\overline{p} \\ \notag
 & +(a\overline{b}+\overline{a}b-2a\overline{a})^2
   +8\big(a\overline{b}+\overline{a}b
   -2(\overline{a}a+\overline{b}b-1)\big)=0.
\end{align}
(Here, we use Risa/Asir, a symbolic computation system, for computing
 the above resultant.)
Substituting $p = a + t(b - a)i$ into \eqref{eq:resul} gives
\begin{align}\label{eq:T}
 & \big((a\overline{b} - \overline{a}b)^2 +
  4(\overline{a} - \overline{b})(a - b)\big)t^2
  +2(a\overline{b} - \overline{a}b)(a\overline{b} + \overline{a}b - 2)it
  \\ \notag
 & -\big((a\overline{b} + \overline{a}b)^2 -
      4(a\overline{a} + b\overline{b} - 1)\big) = 0.
\end{align}

Since $(a\overline{b}-\overline{a}b)$ equals
$ 2\,\mbox{Im}\,(a\overline{b}) i$, the number
$(a\overline{b} - \overline{a}b)$ is purely imaginary.
Therefore, equation \eqref{eq:T} is the quadratic
equation with the real coefficients whose discriminant $ D $
satisfies $ D=64|a-b|^2(1-|a|^2)(1-|b|^2)>0 $ and hence  \eqref{eq:T} has two real solutions.
Here, we remark that the smaller the inradius $ \sqrt{t^2+|a b|^2} $ of the circle centered at $ c $, the larger the inscribed angle with respect to the common chord
$ [a,b] $.
Hence, the solution $ t $ of the smaller absolute value
is the one that yields $ v_{\mathbb{B}^2}(a,b) $ as the maximal
inscribed angle.
\end{proof}

\subsection*{Acknowledgments. }
The first author was partially supported by JSPS KAKENHI
Grant Number JP19K03531.
The research of the second author was supported by a grant from Vilho, Yrj\"{o} and Kalle V\"{a}is\"{a}l\"{a} fund.

\subsection*{Conflict of interest.}

The authors declare that they have no conflict of interest.

\bibliographystyle{siamplain}


\end{document}